\documentclass[11pt]{amsart}
\usepackage{amsmath,amssymb,eucal,graphicx}
\usepackage{amsthm}
\usepackage{hyperref}
\usepackage{float}
\usepackage[section]{placeins}

\usepackage{hyperref}

\usepackage[english]{babel}
\usepackage{graphicx}
\usepackage {enumitem} 
\usepackage{tikz}
\usepackage{caption} 
\usepackage{subcaption}
\usepackage{float} 
\usepackage{multicol}

\usepackage[all,cmtip,matrix,arrow]{xy}

\addto\extrasenglish{

}

\theoremstyle{plain}
\newtheorem{Theorem}{Theorem}[section]

\newtheorem{Corollary}{Corollary}[Theorem]

\newtheorem{Lemma}{Lemma}[section]

\newtheorem{LemCorollary}{Corollary}[Lemma]

\newtheorem{Proposition}{Proposition}[section]

\theoremstyle{definition}
\newtheorem{Definition}{Definition}[section]

\theoremstyle{remark}
\newtheorem*{Remark}{Remark}

\theoremstyle{plain}
\newtheorem{Remarkno}{Remark}[section]

\newcommand{\Pro}{\mathbb{P}}
\newcommand{\C}{\mathbb{C}}
\newcommand{\Z}{\mathbb{Z}}

\newcommand{\Pic}{\mathrm{Pic}}
\newcommand{\Aut}{\mathrm{Aut}}

\title{Equivariant birational geometry of quintic del Pezzo surface}
\author{Jonas Wolter}    
\allowdisplaybreaks
\begin{document}
	
\maketitle

\begin{abstract}
	In this paper we prove that there are exactly two $G$-minimal surfaces which are $G$-birational to the quintic del Pezzo surface, where $G \cong C_5 \rtimes C_4$. These surfaces are the quintic del Pezzo surface itself and the surface $\Pro^1 \times \Pro^1$.
\end{abstract}

\section{Introduction}
The study of the finite subgroups of the \emph{Cremona Group} is classical, but the first serious treatment has been done by Igor V. Dolgachev and Vasily A. Iskovskikh at the beginning of this century, starting with Iskovskikh's paper \cite{Isko}. In their seminal work \cite{Do-Is} all finite subgroups of the Cremona group $\mathrm{Cr}_2(\C)$ are classified up to isomorphism. In the section "What is left" in \cite{Do-Is} it is stated that not all conjugacy classes of $\mathrm{Cr}_2(\C)$ are known and that a finer description of the the conjugacy classes would be desirable.\\
 Let us recall from \cite{Isko} that two subgroups of the Cremona Group given by the biregular actions of a finite group $G$ on two rational surfaces are conjugate if there exist a $G$-birational map $S_1 \dashrightarrow S_2$. By general theory such a map can be factorised into elementary links \cite{Isk_Factor}. In this paper we will contribute to the open questions from \cite{Do-Is} by proving:
  
\begin{Theorem}
	\label{Main_theorem}
	Let $S_5$ be the smooth del Pezzo surface of degree $5$, and let $G_{20}\cong C_5 \rtimes C_4$ be a subgroup of order $20$ in  $\Aut\left(S_5\right)$. Then $\mathrm{Pic}^{G_{20}}(S_5)=\mathbb{Z}$ and
	\begin{enumerate}[label=\arabic*)]
		\item $S_5$ is not $G_{20}$-birational to any conic bundle,
		\item there exists a unique $G$-minimal del Pezzo surface which is \\$G_{20}$-birational to $S_5$, that is $\mathbb{P}^1 \times \mathbb{P}^1$,
		\item the group of $G_{20}$-birational automorphisms is given by\\ $\mathrm{Bir}^{G_{20}}(S_5)=C_2 \times G_{20}$.
	\end{enumerate}
	Here $C_n$ is a cyclic group of order $n$. 	It should be noticed that there are no $G$-conic fibrations birational to $S_5$.
\end{Theorem}

In the notation of  \cite{ahmadinezhad2015birationally} we can say that $S_5$ is $G_{20}$-solid.
\begin{Remark}
	In the proof of \autoref{Main_theorem} we will also see that the only smooth del Pezzo surfaces $G$-birational to $S_5$ are $\mathbb{P}^1 \times \mathbb{P}^1$ and the Clebsch cubic surface. But the latter is not $G_{20}$-minimal, i.e. $\Pic^{G_{20}}(\widetilde{S}) \neq \Z$. Indeed we will show in \autoref{Prop_Picard_Z2_Clebsch}, that its $G_{20}$-invariant Picard group is $\Z^2$.
\end{Remark}

Throughout this paper we assume all varieties to be complex and projective. For all notation in birational geometry, such as \emph{$G$-biregular}, we use the conventions introduced in \cite{Do-Is}.

\section{$G$-Sarkisov links}
\label{Sec_GSarkisov}
We will dedicate this section to the introduction of the notion of \emph{$G$-Sarkisov links} where $G$ is a finite group. For simplicity we will only consider the dimension 2 here. For a more detailed study see \cite{Corti}. This language will allow us to state \autoref{Main_theorem} in a more precise and technical way. We will firstly define a \emph{$G$-Mori fibre space}.

\begin{Definition}
	\label{Defn_GMori}
	A 2-dimensional $G$-Mori fibre space is 
	\begin{itemize}
		\item[$\mathrm{DP}$:] a smooth $G$-minimal del Pezzo surface $S$, i.e. $\Pic^G(S)=\Z$.
		\item[$\mathrm{CB}$:] a $G$-conic bundle, i.e. a $G$-equivariant morphism $\pi: S \to \Pro^1$, where $S$ is a smooth surface and the general fibre of $\pi$ is $\Pro^1$, such that $\Pic^G(S)=\Z^2$.
	\end{itemize}
\end{Definition}

The main result about 2-dimensional $G$-Sarkisov link is the following:

\begin{Theorem}[\cite{Corti}]
	Let $S$, $S'$ be 2-dimensional $G$-Mori fibre spaces and let $\chi: S \dashrightarrow S'$ be a non-biregular $G$-birational map. Then $\chi$ is a composition of elementary links known as $G$-Sarkisov links.
\end{Theorem}

There are 5 different $G$-Sarkisov links of dimension 2 which are described below. The first type is given by

	\begin{align}
	\tag{I}
	\xymatrix{
		&\widehat{S}\ar@{->}[ld]_{\alpha}\ar@{->}[dr]^{\beta} \\
		S &&S'
	}	\label{Sarkisov1}
	\end{align}
	where $S$ and $S'$ are $G$-minimal del Pezzo surfaces and $\alpha$ and $\beta$ are blow ups of $G$-orbits in $S$ and $S'$ respectively. The second type is given by
	\begin{align}
		\tag{II}
	\xymatrix{
		&\widehat{S}\ar@{->}[ld]_{\alpha}\ar@{->}[dr]^{\beta} \\
		S &&\Pro^1
	}	\label{Sarkisov2}	
	\end{align}
	where $S$ is a $G$-minimal del Pezzo surface, $\alpha$ is a blow up of a $G$-orbit and $\beta$ is a $G$-conic bundle. The third type is given by
	\begin{align}
		\tag{III}
	\xymatrix{
		&\widehat{S}\ar@{->}[ld]_{\alpha}\ar@{->}[dr]^{\beta} \\
		\Pro^1 &&S'
	}	\label{Sarkisov3}
	\end{align}
	where $S'$ is a $G$-minimal del Pezzo surface, $\beta$ is a blow up of a $G$-orbit and $\alpha$ is a $G$-conic bundle.	We shall notice that this is the inverse link of type \eqref{Sarkisov2}. The fourth type is given by
\begin{align}
	\tag{IV}
	\xymatrix{
		&\widehat{S}\ar@{->}[ld]_{\alpha}\ar@{->}[dr]^{\beta} \\
		\Pro^1 && \Pro^1
	}	\label{Sarkisov4}
	\end{align}
where $\alpha$ and $\beta$ are $G$-conic bundles. Finally, the fifth type is given by

\begin{align}
\tag{V}
\xymatrix{
	&\widehat{S}\ar@{->}[ld]_{\alpha}\ar@{->}[dr]^{\beta} \\
	S\ar@{->}[d]_{\pi} && S'\ar@{->}[d]^{\pi'} \\
	\Pro^1\ar@{=}[rr] && \Pro^1
} \label{Sarkisov5}
\end{align}
where $S$ and $S'$ are not $G$-minimal del Pezzo surfaces and $\alpha$ and $\beta$ are blow ups of $G$-orbits in $S$ and $S'$ respectively. Additionally, $\pi$ and $\pi'$ are $G$-conic bundles and we call the whole link an \emph{elementary transformation} of $G$-conic bundles (see \cite{Isk_Factor}). This diagram commutes.\\
The notion of $G$-Sarkisov links is a good way to replace the technical result of the Noether-Fano Inequality (see \cite{Do_CAG} and \cite{Isk-NFI}).

\begin{Remarkno}
	\label{Rem_Glinks}
	It follows from the definition of $G$-links that $\widehat{S}$ is a del Pezzo surface, if $S$ is a del Pezzo surface. Thus in the links of type  \eqref{Sarkisov1}, \eqref{Sarkisov2}, \eqref{Sarkisov3} and \eqref{Sarkisov4}, the surface $\widehat{S}$ is a del Pezzo surface.
\end{Remarkno}

Using the notion of $G$-Sarkisov links we are able to restate \autoref{Main_theorem}.

\begin{Theorem}

	\label{MaintechnicalResult}
	Let $S_5$ be the smooth del Pezzo surface of degree 5, and let $G_{20}\cong C_5 \rtimes C_4$ be a subgroup of order 20 in  $\Aut\left(S_5\right)$. Then $\mathrm{Pic}^{G_{20}}(S_5)=\mathbb{Z}$ and the following assertion holds.
	\begin{enumerate}[label=\arabic*)]
		\item There exist a unique $G_{20}$-Sarkisov link that starts at $S_5$. It is given by
	
	\begin{align}
			\tag{$1$}
			\label{MainTechResult1}
			\begin{split}
			\xymatrix{
			&\widetilde{S}\ar@{->}[ld]_{\pi}\ar@{->}[dr]^{\sigma} \\
			S_5 &&\Pro^1 \times \Pro^1}	
			\end{split}
	\end{align}
	where $\pi$ is the blow up of the unique $G_{20}$-orbit of length 2 in $S_5$, $\sigma$ is a blow up of one of two $G_{20}$ orbits of length 5 and $\widetilde{S}$ is the Clebsch cubic surface.
	
	\item Let $\Pro^1 \times \Pro^1$ be equipped with the $G_{20}$-action coming from \eqref{MainTechResult1}. Then the only $G_{20}$-Sarkisov links starting from $\Pro^1 \times \Pro^1$ are the inverse \\ of \eqref{MainTechResult1}, and 
	\begin{align}
	\tag{2}
	\label{MainTechResult2}
	\begin{split}
			\xymatrix{
		&\widetilde{S}\ar@{->}[ld]_{\gamma}\ar@{->}[dr]^{\delta} \\
		\Pro^1 \times \Pro^1 &&S_5
	}	
	\end{split}
	\end{align}

	where $\gamma$ is the blow up of another $G_{20}$-orbit of length 5, $\widetilde{S}$ is the Clebsch cubic surface and $\delta=\pi$ is the blow up of the unique \\ $G_{20}$-orbit of length 2 in $S_5$.
	\end{enumerate}
Combining \eqref{MainTechResult1} and \eqref{MainTechResult2} yields a non-biregular $G_{20}$-birational map $S_5 \dashrightarrow S_5$. 
\end{Theorem}
These links were constructed and described numerically by Dolgachev and Iskovskikh in Proposition 7.13 in \cite{Do-Is} but for our purposes we reconstruct them here and will fill in the details for these links in the paper.

\section{Motivation}
In this section we want to motivate \autoref{Main_theorem}. There are various different starting points to investigate conjugacy in the Cremona group. We decided to start our research on del Pezzo surfaces. Those surfaces have been introduced by Pasquale del Pezzo in the late 18th century and since then various ways of studying them have been encountered. \\
For our purposes we will understand a del Pezzo surface of degree $d$, denoted by $S_d$, as the blow up of $\Pro^2$ in $9-d$ points in general position. To start our investigation of conjugacy classes of the Cremona group we need to introduce the notion of \emph{$G$-birational (super-) rigidity}.

\begin{Definition}
	\label{Defn_rigid}
Let $S$ be a smooth del Pezzo surface and $G \subset \Aut(S)$ be a finite group, such that $\Pic^G(S)=\Z$. We say $S$ is \emph{$G$-birationally rigid}
\begin{enumerate}[label=(\arabic*)]
	\item if $S$ is $G$-birational to any $G$-minimal del Pezzo surface $S'$, then $S'$ is $G$-biregular to $S$, and
	\item $S$ is not $G$-birational to any $G$-conic bundle.
\end{enumerate} 
Condition (1) is equivalent to saying that for any  birational $G$-map \\$\chi : S \dashrightarrow S'$, where $S'$ is a $G$-minimal del Pezzo surface, there exist a $G$-birational automorphism $\theta: S \dashrightarrow S$ such that $\chi \circ \theta$ is a $G$-isomorphism.
\end{Definition}

	\autoref{Defn_rigid} means, that the only $G$-Sarkisov links starting in $S$ are of the form 
	\begin{align}
	\tag{$\bigstar$}
	\label{ComDiagram_Grigid}
	\begin{split}
		\xymatrix{
		&\widehat{S}\ar@{->}[ld]_{\alpha}\ar@{->}[dr]^{\beta}\\
		S\ar@{-->}[rr]_{\phi}&&S
	}
	\end{split}
	\end{align}
 where $\alpha$ and $\beta$ are blow ups of $G$-orbits.

\begin{Definition}
	\label{Defn_supperrid}
	Let $S$ be a smooth del Pezzo surface and $G \subset\Aut(S)$ be a finite group, such that $\Pic^G(S)=\Z$. The surface $S$ is \emph{$G$-birationally superrigid}, if it is $G$-birationally rigid and $\text{Bir}^G(S)=\Aut^G(S)$.
\end{Definition}
\noindent 	\autoref{Defn_supperrid} means that there are no $G$-Sarkisov links starting at $S$.\\
	
With these definitions in hand we are able to state.

\begin{Theorem}[\cite{Do-Is}]
	\label{Thm_Orbitlengthrigid}
	Let $S$ be a smooth del Pezzo surface of degree d, that is $K_S^2=d$, and let $G \subseteq \Aut(S)$ such that $\Pic^G(S)=\Z$. Then the following assertion holds.
	\begin{enumerate}
		\item If $S$ does not contain a $G$-orbit of length less then $d$, then $S$ is \\$G$-birationally superrigid.
		\item If $S$ does not contain a $G$-orbit of length less then $d-2$, then $S$ is $G$-birationally rigid.
	\end{enumerate}
\end{Theorem}

\begin{proof}[Sketched Proof]
	Assume $\widehat{S}$ is a smooth del Pezzo surface. Then by \autoref{Rem_Glinks} $\alpha:S \to \widehat{S}$  is a blow up of a $G$-orbit of length less then $d$, because $K_{\widehat{S}}>0$. This proves \emph{(1)}.\\
	If there is a $G$-orbit of length $d-1$, the blow up of this orbit is $\widehat{S}=S_1$, the del Pezzo surface of degree 1, so we can use the Bertini involution there. Similarly if there exists a $G$-orbit of length $d-2$, the del Pezzo surface of degree 2, we can blow up this orbit to obtain $\widehat{S}=S_2$, and we can use the Geiser involution. This prove \emph{(2)}.
\end{proof}

From \autoref{Thm_Orbitlengthrigid} we can immediately deduce the following Corollary.

\begin{Corollary}[\cite{Do-Is}]
	\label{Thm_DP1}
	Let $S$ be a smooth del Pezzo surface of degree $d <3$, and let $G \subseteq \Aut\left(S\right)$ be a finite group such that $\mathrm{Pic}^G(S)=\mathbb{Z}$. If $S$ is of degree $1$, then $S$ is $G$-birationally superrigid. If $S$ is of degree 2 or 3, then $S$ is $G$-birationally rigid.
\end{Corollary}

This result is known for quite some time and was implicitly proven by Segre in 1943 and Manin in 1962. For proofs of \autoref{Thm_Orbitlengthrigid} and \autoref{Thm_DP1}  see section 7.1 in \cite{Do-Is}. The proof of \autoref{Thm_Orbitlengthrigid} easily implies

\begin{Theorem}[\cite{Do-Is}]
	\label{Thm_fixedpoints}
	Let $S$ be a smooth del Pezzo surface of degree $4$, and let $G \subset \Aut\left(S\right)$ be a finite group such that $\mathrm{Pic}^G(S)=\mathbb{Z}$. Then 
	\begin{enumerate}[label=\arabic*)]
		\item if there are no $G$-fixed points, then $S$ is $G$-birational rigid,
		\item if there exists a $G$-fixed point, then there exists a $G$-Sarkisov link
		\begin{align*}
		\begin{split}
			\xymatrix{
			&\widetilde{S}\ar@{->}[ld]_{\alpha}\ar@{->}[dr]^{\beta}\\
			S_4\ar@{-->}[rr]_{}&&\Pro^1
		}
		\end{split}
		\end{align*}
		where $\alpha$ is the blow up of a $G$-orbit, $\widetilde{S}$ is a smooth cubic surface and $\beta$ is a conic bundle. 
	\end{enumerate} 
\end{Theorem}

In this paper we are mostly interested in $G$-birational rigid del Pezzo surfaces or those which are \emph{close} to them. By \emph{close} we mean that these are del Pezzo surfaces which are not $G$-birational to any conic bundle (in the language of \cite{ahmadinezhad2015birationally} these are $G$-solid del Pezzo surfaces).\\

Following from \autoref{Thm_DP1} and \autoref{Thm_fixedpoints} we will investigate links starting from the smooth del Pezzo surface of degree 5, which we will call $S_5$, in this paper.
It is well known that 
$$\Aut(S_5) \cong \mathfrak{S}_5,$$
 the symmetric group of 5 elements. A proof is provided in  \cite{Blanc_thesis} . If we want $S_5$ to be a $G$-minimal surface (i.e. $\mathrm{Pic}^G(S_5)=\mathbb{Z}$), we require $G$ to be one of the following (see Theorem 6.4 in \cite{Do-Is}):
\begin{itemize}
 \setlength\itemsep{0.5em}
	\item the symmetric group $\mathfrak{S}_5$ of 5 elements of order 120;
	\item the alternating group $\mathfrak{A}_5$ of 5 elements of order 60;
	\item the semidirect product $G_{20} \cong C_5 \rtimes C_4$ of order 20;
	\item the dihedral group $D_{10}$ of order 10;
	\item the cyclic group $C_5$ of order 5.
\end{itemize}
\vspace{1mm} 
For $\mathfrak{S}_5$ and $\mathfrak{A}_5$ the quintic del Pezzo surface is $G$-birationally superrigid (see \cite{Ban-Tok}). For $C_5$ there exist a $G$-birational map from $S_5$ to $\mathbb{P}^2$ (see \cite{Bau-Bla}) such that $C_5$ has a fixed point there (see \autoref{Lem_C5P2}). The construction of this map can be generalised for $D_{10}$ which is done in \autoref{Prop_D10DP5}. Hence these groups are better adressed when studying the $G$-equivariant birational geometry of $\Pro^2$. This has been done in \cite{Sako}.
We shall also notice that $S_5$ is not $G$-solid in this case.\\
In this paper we will therefore focus on the group $G_{20} \cong C_5 \rtimes C_4$ as a subgroup of $\Aut(S_5) \cong \mathfrak{S}_5$, which is also known as as the general affine group of degree one over the field with five elements, denoted by $\mathrm{GA}(1,5)$.

\section{The Quintic del Pezzo surface}
\label{Sec_DP5}

In the proof of \autoref{MaintechnicalResult} we will investigate the existence of $G_{20}$-equivariant birational maps between quintic del Pezzo surface, denoted \\by $S_5$ and the surface $\mathbb{P}^1 \times \mathbb{P}^1$.  First we need to understand the action of $G_{20}$ on $S_5$. To do this we use the result from \cite{Bau-Bla}.

\begin{Lemma}[\cite{Bau-Bla}]
	\label{Lem_C5P2}
	There is a $C_5$-birational map $\phi$ (i.e. a $C_5$-Sarkisov link) between $S_5$ and $\Pro^2$  given by the $C_5$-commutative diagram
	\begin{align}
		\tag{$\blacktriangle$}
		\label{Fig_DP5P2}
	\begin{split}
	\xymatrix{
		&S_4\ar@{->}[ld]_{\alpha}\ar@{->}[dr]^{\beta}\\
		S_5 \ar@{-->}[rr]_{\phi}&&\mathbb{P}^2 
	}		
	\end{split}
	\end{align}
	Here $\alpha$ is the blow up of a $C_5$-fixed point in $S_5$, and $\beta$ is the blow up of $5$ points in $\Pro^2$ which form a $C_5$-orbit. $S_4$ is a quartic del Pezzo surface.
\end{Lemma}

\begin{proof}
	For the proof we will start with $\Pro^2$ and invert the link \eqref{Fig_DP5P2}. Consider $C_5$ as a subgroup of $\Aut\left(\Pro^1\right)\cong \text{PGL}_2(C)$. There exist a $C_5$-equivariant Veronese embedding $\Pro^1 \hookrightarrow \Pro^2$ which defines a faithful action of $C_5$ on $\Pro^2$ such that there exists a $C_5$-invariant conic $K \subseteq \Pro^2$ (that is the image of $\Pro^1$). Thus we can blow up the $C_5$-orbit of length 5 on this conic to obtain the quartic del Pezzo surface, denoted by $S_4$.\\
	 If we contract the the proper transform of $K$ there we get the unique quintic del Pezzo surface. Since $C_5 \subseteq \mathfrak{S} \cong \Aut(S_5)$ is unique up to conjugation the composition of the two described maps yields the desired link $\phi$.\\[0.1cm]
	
	In more elementary terms we may say that five points $P_1,...,P_5$ in general position in $\mathbb{P}^2$ always lie on a unique conic $K$. Then the group $C_5$ fixes two points $A_1$, $A_2$ on a conic \cite{Blanc_thesis} (i.e. the line through these two points is $C_5$-invariant). Additionally it fixes a point $B \in \Pro^2$ which does not lie on the conic. The blow up $\alpha$ of $P_1, ..., P_5$ does not affect $B$, neither does the contraction $\beta$. Thus there is a point 
	$$Q_2 = \phi^{-1}\left(B\right) \in S_5,$$
	 which is fixed by $C_5$. We know that $\alpha^{-1}\left(K\right)$ is a $\beta$-exceptional curve in $S_4$. After the contraction $\beta$, we have $$\phi^{-1}\left(A_1\right)=\phi^{-1}\left(A_2\right)=Q_1,$$
	  which is another fixed point of $C_5$ in $S_5$. Thus we know that for \\ $C_5 \subseteq \Aut(S_5)$ there exist two $C_5$ fixed points $Q_1$ and $Q_2$. We shall mention that all other orbits are of length 5.
\end{proof}



From the proof of \autoref{Lem_C5P2} we can easily deduce.

\begin{LemCorollary}
	\label{Prop_D10DP5}
	There is a $D_{10}$-birational map $\phi$ (i.e. a $D_{10}$-Sarkisov link) between $S_5$ and $\Pro^2$  corresponding to the $D_{10}$-commutative diagram \eqref{Fig_DP5P2}. 
\end{LemCorollary}

\begin{proof}
	In the same way as in the proof of \autoref{Lem_C5P2} we can construct the inverse link of \eqref{Fig_DP5P2}. Furthermore the action of $D_{10}\subseteq \Aut\left(\Pro^1\right)$ lifts to an action on $\Pro^2$. Then we can use the same argument as before.\\[1mm]
	In the notation of the proof of \autoref{Lem_C5P2} we may say that the action of $D_{10}$ on $\Pro^2$ interchanges the points $A_1$ ad $A_2$ but fixes the point $B$. Thus we can use the same link $\phi$ as in \autoref{Lem_C5P2} and by the same argument as above $D_{10}$ fixes points $Q_1$ and $Q_2$ in $S_5$.
\end{proof}

 We are now in the position to investigate orbits of small length $r < 5$ of the $G_{20}$-action on $S_5$.\\
	We want to proceed in a similar way as in \cite{Isko} which means that we need to classify all $G$-orbits of length $r < \deg(S_5) = \text{K}_{S_5}^2=5$. Then we will concentrate on those orbits of which the points are in general position, because this is a necessary condition for the existence of links starting from the surface $S_5$.

\begin{Remark}
	 We say that points of an orbit are in general position if the blow up of $S_5$ in this orbit is a del Pezzo surface again. 
\end{Remark}

\begin{Lemma}
	\label{Lem_OrbitsDP5}
	There is a unique $G_{20}$-orbit of length $r <5$  on $S_5$. It is the orbit of length $r=2$ consisting of the points $Q_1$ and $Q_2$.
\end{Lemma}

\begin{proof}
	Let us consider all possible lengths for orbits.
	\begin{itemize}
		\item [$r=1$:] Such an orbit does not exist. Assume it does. By \autoref{Lem_C5P2} and \autoref{Prop_D10DP5} this point can only be $Q_1$, because if all of $G_{20}$ fixes it, the normal subgroups $C_5$ and $D_{10}$ fix it in particular.
		Hence the link \eqref{Fig_DP5P2} yields $G_{20}$-equivariant link from $S_5$ to $\Pro^2$. This means that $G_{20}$ acts on $\Pro^2$ and preserves the conic $K$. This implies that $G_{20}$ acts faithfully on $K\cong \Pro^1$, but this is clearly a contradiction. Hence no orbit of length $r=1$ exists.  
		
		\item [$r=2$:] $\left\{ Q_1,Q_2\right\}$ is such an orbit.	We know that $G_{20}$ has $D_{10}$ as a normal subgroup. If we consider the action of $D_{10}$ on $S_5$, then \autoref{Prop_D10DP5} tells us that there is indeed a unique orbit of length 2 which is the orbit $\left\{Q_1,Q_2 \right\}$. 
		
		\item [$r=3$:] Such an orbit does not exist because $3\nmid 20 = |G_{20}|$, which is required by the orbit-stabilizer theorem.
		
		\item [$r=4$:] Such an orbit does not exist. If there was such an orbit the stabilizer would satisfy $\operatorname{Stab}_G=C_5$ but we know that $C_5$ actually fixes the same points as $D_{10}$ by \autoref{Prop_D10DP5} and hence the stabilizer would actually be $D_{10}$ which can not give an orbit of length 4.
	\end{itemize}
This proves \autoref{Lem_OrbitsDP5}.
\end{proof}

\autoref{Lem_OrbitsDP5} implies that the only possible $G_{20}$-Sarkisov link starting from $S_5$ consists of a blow up of the described orbit of length $r=2$.
\begin{Lemma}
	\label{Lem_SmoothBlowupS5}
	The blow up of $Q_1$ and $Q_2$ in $S_5$ yields a smooth del Pezzo surface $\widetilde{S}$.
\end{Lemma}

\begin{proof}
	We need to prove that $-K_{\widetilde{S}}$ is ample. This is equivalent to saying that $Q_1$ and $Q_2$ neither lie on the $(-1)$-curves nor in an exceptional conic in $S_5$. We prove this by contradiction. For this we will consider different cases.
	\begin{enumerate}[label=\arabic*)]
		\item We first prove that there are no $-1$ curves containing $Q_1$ or $Q_2$.\\
		Assume $Q_1$ lies on one of the 10 exceptional curves in $S_5$. Clearly $Q_2$ needs to lie on such a curve as well. If they lie on two different exceptional curves these two are interchanged by the group action of $G_{20}$. This contradicts the fact that $\Pic^{G_{20}}(S_5)=\Z$.\\
		Similarly we may assume that $Q_1$ lies on one of the intersections of two exceptional curves. Again this contradicts $\Pic^{G_{20}}(S_5)=\Z$. So indeed $Q_1$ and $Q_2$ do not lie on the $(-1)$-curves in $S_5$ which proves that the blow up of these two points yields another del Pezzo surface.
		
		\item It remains to show that $Q_1$ and $Q_2$ are not contained in an exceptional conic $S_5$.
		 There are 5 classes of conic in $S_5$ and each of them has self intersection $C^2=0$. Going through all these cases in detail on can show that $Q_1$ and $Q_2$ either lie on one line which we ruled our previously or the cannot lie on one conic. Due to heavy computational work we omit the different cases at this point. Thus \autoref{Lem_SmoothBlowupS5} is proven.	
	\end{enumerate}
.
\end{proof}
 The resulting surface of this blow up will have degree $5-2=3$, so it is a cubic surface. The only smooth cubic surface with a $G_{20}$-action is the \emph{Clebsch cubic surface} (this was proved in \cite{Hosoh}) which we will investigate in the next section.

\section{The Clebsch cubic surface}
\label{Sec_ClebschCubic}
\autoref{MaintechnicalResult} states that the only $G$-Sarkisov links starting from the quintic del Pezzo surface $S_5$ are of the form
\begin{align}
\tag{$\blacklozenge$}
\label{MainDiagram}
\begin{split}
\xymatrix{
	&\widetilde{S}\ar@{->}[ld]_{\pi}\ar@{->}[dr]^{\sigma}\\
	S_5 \ar@{-->}[rr]_{\psi}&& \Pro^1 \times \Pro^1
}
\end{split}
\end{align}
From \autoref{Lem_OrbitsDP5} we know that $\pi$ is the blow up of the unique $G_{20}$-orbit $\left\{Q_1, Q_2\right\}$ of length 2. Hence $\widetilde{S}$ is the Clebsch cubic surface, which is defined as follows:
\begin{Definition}
	\label{Defn_Clebsch}
		The \emph{Clebsch cubic surface}, denoted by $\widetilde{S}$, is a cubic given by two defining equations in $\mathbb{P}^4$:
	\begin{align*}
	\begin{cases}
	x_0+x_1+x_2+x_3+x_4=0;\\
	x_0^3+x_1^3+x_2^3+x_3^3+x_4^3=0.
	\end{cases}
	\end{align*}
\end{Definition}

\begin{Remarkno}
	\label{Prop_Picard_Z2_Clebsch}
	In \cite{Do-Is} it is shown that $\Pic^{G_{20}}(\widetilde{S})\neq \Z$. The link \eqref{MainDiagram} proves that in fact,
	$
	\Pic^{G_{20}}(\widetilde{S})=\mathbb{Z}^2.
	$

\end{Remarkno}

Now it is well known that the automorphism group of the Clebsch cubic surface, which we will call $\widetilde{S}$, is $\Aut(\widetilde{S})=\mathfrak{S}_5$. Thus the action of $G_{20}$ can be described very explicitly by understanding $G_{20}$ as a subgroup of $\mathfrak{S}_5$ acting by permutation on the coordinates of this surface.\\
We know that all representation of $G_{20}$ are conjugate to each other and thus we will use a generation by $\sigma_{(12345)}$ and $\sigma_{(2354)}$, where we use the notation introduced in \cite{Blanc_thesis}. Considering orbits of length 4 on $\widetilde{S}$ we obtain

\begin{Lemma}
	\label{Lem_OrbClebsch4}
	There is a unique orbit of length $4$ of the $G_{20}$-action on the Clebsch cubic surface given by the points 
	\begin{align*}
	\mathcal{O}=\bigg\{\left(1:\zeta:\zeta^2:\zeta^3:\zeta^4\right),\left(1:\zeta^2:\zeta^4:\zeta:\zeta^3\right),\\ \left(1:\zeta^3:\zeta:\zeta^4:\zeta^2\right),\left(1:\zeta^4:\zeta^3:\zeta^2:\zeta\right)\bigg\},
	\end{align*}
	with $\zeta$ being a primitive fifth root of unity.
	
\end{Lemma}

\begin{proof}
	An orbit of length 4 has the stabilizer $\operatorname{Stab}_G = K \cong C_5$ which is isomorphic to the group generated by  $\sigma_{(12345)}$, which is the unique subgroup of $G_{20}$ isomorphic to $C_5$. It has exactly the fixed points as stated in \autoref{Lem_OrbClebsch4}, which are obtained by straightforward calculations.
	It is easy to verify that these 4 points lie indeed on $\widetilde{S}$ and form an orbit of length 4.
\end{proof}

The orbits of lengths 5 are a bit more sophisticated. 
\begin{Lemma}
	\label{Lem_OrbClebsch5}
	There are three orbits of length $5$ of the $G_{20}$-action on the Clebsch cubic surface given by:
	\begin{align*}
	\mathcal{O}_1 = \bigg\{V_1=\left(0:-1:1:1:-1\right), V_2=\left(-1:0:-1:1:1\right),\\V_3=\left(1:-1:0:-1:1\right),
	V_4=\left(1:1:-1:0-1\right),\\V_5=\left(-1:1:1:-1:0\right)\bigg\},\\
	\mathcal{O}_2 = \bigg\{U_1=\left(0:-i:-1:1:i\right), U_2=\left(i:0:-i:-1:1\right),\\U_3=\left(1:i:0:-i:-1\right),
	U_4=\left(-1:1:i:0:-i\right),\\U_5=\left(-i:-1:1:i:0\right)\bigg\},\\
	\mathcal{O}_3 = \bigg\{W_1=\left(0:i:-1:1:-i\right), W_2=\left(-i:0:i:-1:1\right),\\W_3=\left(1:-i:0:i:-1\right),
	W_4=\left(-1:1:-i:0:i\right),\\W_5=\left(i:-1:1:-i:0\right)\bigg\}.
	\end{align*}
\end{Lemma}

\begin{proof}
	An orbit of length 5 has the stabilizer $\operatorname{Stab}_G=H\cong C_4$ in $G_{20}$. There are five subgroups of $G_{20}$ which are isomorphic to $C_4$. Let $H \cong C_4$ be the subgroup generated by $\sigma_{(2354)}$. Then $H$ fixes four points in $\mathbb{P}^4$ with $\sum\limits\limits_{i=1}^{5}x_i=0$ which are:
	
	\begin{align*}
	R_1=\left(0:-1:1:1:-1\right),R_2=\left(0:-i:-1:1:i\right),\\
	R_3=\left(0:i:-1:1:-i\right),R_4=\left(-4:1:1:1:1\right),
	\end{align*}
	whereas the $R_4$ does not lie on $\widetilde{S}$ because the cubes of the coordinates do not sum to zero. Again it is easy to verify that the points $\left(R_1,...,R_4\right)$ are indeed fixed points. Acting by an element of order $5$, we obtain fixed points corresponding to the action of $\sigma_{(12345)}$ on the coordinates of $R_i$. Thus we deduce, that there are three orbits of length 5 on $\widetilde{S}$ as stated in \autoref{Lem_OrbClebsch5}.
\end{proof}

We shall notice that $R_2$ and $R_3$ lie on the line $x_1+x_4=x_2+x_3=0$. Generalising this we make the following important observation.

\begin{LemCorollary}
	\label{Cor_5skewlinesClebsch}
	The points $U_i \in \mathcal{O}_2$ and $W_i \in \mathcal{O}_3$ respectively lie on one of the $27$ real lines on the Clebsch cubic surfaces. These $5$ resulting lines in the link are
	\begin{enumerate}[label=(\roman{*})]
		\item $L_1:x_1+x_4=x_2+x_3=0$ through $U_1$ and $W_1$.
		\item $L_2:x_0+x_2=x_3+x_4=0$ through $U_2$ and $W_2$.
		\item $L_3:x_0+x_4=x_1+x_3=0$ through $U_3$ and $W_3$.
		\item $L_4:x_0+x_1=x_2+x_4=0$ through $U_4$ and $W_4$.
		\item $L_5:x_0+x_3=x_1+x_2=0$ through $U_5$ and $W_5$.
	\end{enumerate}
It is easy to see that these $5$ lines are disjoint.
\end{LemCorollary}

\begin{proof}
	This is an easy exercise of calculating the lines through each pair of points and comparing it with the lines on the Clebsch cubic, which are well known.
\end{proof}

\autoref{Lem_OrbClebsch4} and \autoref{Lem_OrbClebsch5} allow us to state the main result for this section.

\begin{Proposition}
	\label{Prop_OrbitClebsch}
	Let $\widetilde{S}$ be the Clebsch cubic surface. Then the $G_{20}$-orbits of length $r < 8$ on $\widetilde{S}$ are: 
	\begin{enumerate}[label=\alph*)]
		\item The unique orbit $\mathcal{O}$ described in \autoref{Lem_OrbClebsch4} of length 4.
		\item The three orbits $\mathcal{O}_1, \mathcal{O}_2$ and $\mathcal{O}_3$ described in \autoref{Lem_OrbClebsch5} of length 5.	
	\end{enumerate}
\end{Proposition}

\begin{proof}
	The orbit-stabilizer theorem tells us immediately that orbits of length $r=6$ or $r=7$ can not exist. It remains to show that there are no orbits of length 1 or 2 on $\widetilde{S}$. This follows directly from our description of the orbits but we include computationally explanation, too. An orbit of length 1 would have the whole group $G_{20}$ as its stabilizer. We see immediately that this is not possible because the subgroups $K$ and $H$ generated by $\sigma_{(12345)}$ and $\sigma_{(2354)}$ have completely different fixed points.\\
	By a similar argument there can not be any orbits of length $2$. These would have the subgroup $F \cong D_{10}$ generated by $\sigma_{(12345)}$ and $\sigma_{(25)(34)}$ as its stabilizer. Again it is easy to verify that $F$ has $K\cong C_5$ as a subgroup. On the other hand $F$ has the group generated isomorphic to $C_2$ generated by  $\sigma_{(25)(34)}$ which is a subgroup of $H$ as a subgroup.\\
	 But we have seen that $H$ and $K$ do not have any common fixed points. Hence $F$ can not have fixed points which means that there does not exist an orbit of length 2.  
\end{proof}

\begin{Remark}
	\autoref{Prop_OrbitClebsch} supports the statement of \autoref{Lem_OrbitsDP5}. For the unique orbit $\mathcal{O}$ of length 4 each pair of points lies on one of 27 real lines on the Clebsch cubic. Hence after contracting 2 of them to obtain $S_5$, we are left with an orbit of length 2.\\
	An orbit of length 4 in $S_5$ would lift to a different orbit of length 4 in $\widetilde{S}$, but for the given reason this can not be $\mathcal{O}$, which means that there do not exist orbits of length 4 in the quintic del Pezzo surface.
\end{Remark}

Given \autoref{Cor_5skewlinesClebsch} we may consider the contraction of these 5 lines. 

\begin{Proposition}
	\label{Prop_Contraction5lines}
	The contraction of the $5$ lines $L_1,...,L_5$ described in \autoref{Cor_5skewlinesClebsch} yields the surface $\Pro^1 \times \Pro^1$ and this is the only other contraction that can be conducted apart from the inverse of the blow up from $S_5$.
\end{Proposition}

\begin{proof}
	We know that $\widetilde{S}$ is a del Pezzo surface, so $-K_{\widetilde{S}}$ is ample. \autoref{Rem_Glinks} tells us that the resulting surface of the described contraction will be a del Pezzo surface of degree $3+5=8$, so it can only be $\Pro^1 \times \Pro^1$ or $\mathbb{F}_1$, but $\Pic^{G_{20}}\left(\mathbb{F}_1\right) \neq \Z$, which we require.\\
	In \autoref{Prop_Picard_Z2_Clebsch} we have seen that $\Pic^{G_{20}}(\widetilde{S})=\Z^2$. From this we conclude that there are two external rays in the Mori cone. We have shown that one consists of two lines and the other one of 5. These are the only possible contraction of $\widetilde{S}$.
\end{proof}

\autoref{Prop_Contraction5lines} allows us to state the following lemma about the link \eqref{MainDiagram} which we introduced at the beginning of this section.
\begin{Lemma}
	\label{Lem_LinkClebschP1xP1}
	Considering the desired link \eqref{MainDiagram} from $S_5$ to $\Pro^1 \times \Pro^1$ we know
	\begin{enumerate}[label=\arabic*)]
		\item $\pi$ is the contraction of 2 disjoint lines $E_1, E_2$ in the Clebsch cubic surface (respectively the blow up of $Q_1$ and $Q_2$ in $S_5$)
		\item $\sigma$ is the contraction of 5 disjoint lines $F_1,...,F_5$ in the Clebsch cubic surface (respectively the blow up of 5 points in $\Pro^1 \times \Pro^1$).
		\item The following equations hold for the exceptional divisors:
		\begin{align*}
		\sigma^*(H)=2\pi^*\left(-K_{S_5}\right)- 3(E_1+E_2)\\
		\sum\limits_{i=1}^{5}F_i= 3\pi^*\left(-K_{S_5}\right) - 5(E_1+E_2),
		\end{align*}
		where $-K_{S_5}$ is the anticanonical divisor of $S_5$, $(E_1+E_2)$ are the two $(-1)$-curves of the blow up of $Q_1$ and $Q_2$, $H$ is a divisor of bidegree $(1,1)$ on $\Pro^1 \times \Pro^1$ and $\sum\limits_{i=1}^5F_i$ are the $(-1)$-curves of the blow up of 5 points in $\Pro^1 \times \Pro^1$.
	\end{enumerate}
\end{Lemma}

\autoref{Lem_LinkClebschP1xP1} implies that $\pi(F_i)$ is a smooth twisted cubic curve in $S_5$ and $E_1$ and $E_2$ are smooth twisted cubics of bidegree $(2,1)$ and $(1,2)$ respectively.

\section{The surface $\mathbb{P}^1\times \mathbb{P}^1$}
\label{P1xP1}
The $G_{20}$-action on $\Pro^1 \times \Pro^1$ can not be understood in a way which is a simple as in \autoref{Sec_DP5} or \autoref{Sec_ClebschCubic}. For that reason we will use our previous observations to analyse the $G_{20}$-orbits on $\Pro^1 \times \Pro^1$.

\begin{Lemma}
	\label{Lem_OrbitP1xP14}
	 There is a unique $G_{20}$-orbit  $\mathcal{K}$  of length 4 in $\Pro^1 \times \Pro^1$. The four points are given by the intersections $F_{11} \cap F_{21}, F_{11} \cap F_{22}, F_{12} \cap F_{21}$ and $F_{12} \cap F_{22}$ of the four rulings $F_{11}, F_{12}, F_{21}$ and $F_{22}$.
\end{Lemma}

\begin{proof} The orbit of length 4 described in \autoref{Lem_OrbClebsch4} lies away from the lines of contraction, thus it has an embedding in $\Pro^1 \times \Pro^1$.\\
 Generally 4 points need 8 different lines to describe them. 
 A $G_{20}$-orbit of length 4 has the stabilizer $C_5$.
 But we know that $D_{10} \subset G_{20}$ acts on the rulings of $\Pro^1 \times \Pro^1$ which are copies of $\Pro^1$.
 Hence the $C_5$-action can not split over each of the 4 points (i.e. interchanging the two lines) but fixes the rulings.\\
  For this reason the 4 points on the $G_{20}$-orbit need to lie on the four intersection of four copies of $\Pro^1$ (i.e the rulings of $\Pro^1 \times \Pro^1$), because otherwise it would not be an orbit of length 4. Hence all four points in this orbit lie on two rulings in $\Pro^1 \times \Pro^1$.
\end{proof}

\begin{Remark}
	We could prove \autoref{Lem_OrbitP1xP14} in a different way by considering the orbit of length 4 in the Clebsch cubic surface and considering their configuration there. We can show merely computationally that there exists four conics each passing through exactly one of the four points and not intersecting the $(-1)$-curves. Considering the blow up $\sigma$ we obtain \autoref{Lem_OrbitP1xP14} for the four points on $\Pro^1 \times \Pro^1$. 
\end{Remark}

As in \autoref{Sec_ClebschCubic} the orbits of length 5 are a bit more difficult. 
\begin{Lemma}
	\label{Lem_OrbitsP1xP15}
	There are exactly two $G_{20}$-orbits  $\mathcal{K}_1$ and $\mathcal{K}_2$ of length 5 in $\Pro^1 \times \Pro^1$.
\end{Lemma}

\begin{proof}
	We will now identify $\Pro^1 \times \Pro^1$ as a quadric $\mathcal{Q}$ in $\Pro^3$. We know that $\Pro^1 \times \Pro^1$ has a natural embedding (Segre) into $\Pro^3$. Similar to the Clebsch cubic we can understand $\Pro^3$ as a hyperplane in $\Pro^4$ with $\sum\limits_{i=0}^{4}x_i=0$. Now let $\mathcal{Q} \cong \Pro^1 \times \Pro^1 \subset \Pro^3$ be the quadric given by 
	$$\mathcal{Q}: \sum\limits_{i=0}^{4}x_i=\sum\limits_{i=0}^4x_i^2=0.$$
	
	$G_{20}$ acts on $\mathcal{Q}$ by permutations of coordinates in a similar way as described in the previous section for the Clebsch cubic.

	We found the orbits of length 5 explicitly on $\widetilde{S}$. Observe that the orbits $\mathcal{O}_2$ and $\mathcal{O}_3$ lie in $\mathcal{Q}$, whereas $\mathcal{O}_1$ does not. Hence we may assume $\mathcal{K}_1=\mathcal{O}_2$ and $\mathcal{K}_2=\mathcal{O}_3$.\\
	 Additionally we can check computationally that the points of each of the orbits lie in general position (i.e. no 2 on a line in $\mathcal{Q}$ and no 4 on a plane). Hence we can indeed consider the blow up of each of these orbits which will yield the Clebsch cubic surface and together with \autoref{Lem_LinkClebschP1xP1} shows that this blow up is indeed the inverse link of the described contraction.\\
	Furthermore we see that the orbits $\mathcal{O}_2$ and $\mathcal{O}_3$ are essentially the same orbits, only permuted by complex conjugation. In fact these two orbits are interchanged by an automorphism of a quadric which is proven in \autoref{Thm_G40}.
\end{proof}

Now we can finally state the last Proposition we need for the proof of \autoref{Main_theorem}.

\begin{Proposition}
	\label{Prop_OrbitsP1xP1}
	The only $G_{20}$-orbits of length $r < 8$ on $\Pro^1 \times \Pro^1$ are:
	\begin{enumerate}[label=\alph*)]
		\item The unique orbit $\mathcal{K}$ described in \autoref{Lem_OrbitP1xP14} of length $4$.
		\item The two orbits $\mathcal{K}_1$ and $\mathcal{K}_2$ described in \autoref{Lem_OrbitsP1xP15} of length $5$ where the $5$ points lie in general position.	
	\end{enumerate}
\end{Proposition}

\begin{proof}
	It remains to show that there are no other orbits than the ones described in \autoref{Lem_OrbitP1xP14} and \autoref{Lem_OrbitsP1xP15}.\\
	 Orbits of length 6, 7 or 8 can not exists by the orbit-stabilizer-theorem as $6,7,8 \nmid 20=\left|G_{20}\right|$. Assume there is an orbits of length less than 4. Then an orbit of this length would also exist in the Clebsch cubic surface but \autoref{Prop_OrbitClebsch} tells us, that they do not exist there.
\end{proof}

\section{Proof of \autoref{MaintechnicalResult}}
\label{Sec_proof}
The link \eqref{MainTechResult1} in \autoref{MaintechnicalResult} is the only $G$-Sarkisov links starting from the quintic del Pezzo surface $S_5$. From \autoref{Lem_OrbitsDP5} we know that $\pi$ is the blow up of the unique $G_{20}$-orbit $\left\{Q_1, Q_2\right\}$ of length 2.\\
Now \autoref{Cor_5skewlinesClebsch} tells us that there are five disjoint lines on the Clebsch cubic which we can contract to obtain $\Pro^1 \times \Pro^1$ and \autoref{Prop_OrbitsP1xP1} says that we need to consider two different cases for birational maps starting from $\Pro^1 \times \Pro^1$. 

\begin{Lemma} [Orbit of length 4]
	\label{Lem_BlowupOrbit4P1xP1}
	Let $\tau: \widetilde{S} \to \Pro^1 \times \Pro^1$ be the blow up of the four points $P_1,...P_4$ in the orbit $\mathcal{K}$ and let $E_1,...,E_4$ be the corresponding exceptional curves. Then the proper transforms of the described rulings are $(-2)$-curves. This means that  the resulting surface $\widehat{S}$ is not a del Pezzo surface. 
\end{Lemma}
\begin{proof}
 This follows immediately from \autoref{Lem_OrbitP1xP14}.
\end{proof}
\autoref{Lem_BlowupOrbit4P1xP1} tells us that we can not continue from $\Pro^1 \times \Pro^1$ to obtain a $G$-Sarkisov link by blowing up the oribit $\mathcal{K}$ of length 4.

\begin{Lemma}
	\label{Lem_Blowuporbit5P1xP1}
	Let $\tau: \widetilde{S} \to \Pro^1 \times \Pro^1$ be a blow up of the five points $P_1,...P_5$ in the orbits $\mathcal{K}_1$ or $\mathcal{K}_2$ respectively. Then one of the following holds for $\tau$ :
	\begin{enumerate}[label=\alph*)]
		\item $\tau$ is the same as the blow up $\sigma$ described in \eqref{MainDiagram}, so that $\pi \circ \tau^*=\psi^{-1}$.
		\item  $\tau$ is the same as the blow up $\gamma$ described in diagram \eqref{MainTechResult2} of \autoref{MaintechnicalResult}, so that $(\psi \circ \tau^* \circ \pi)=\psi \circ \phi=\chi$ is a $G_{20}$-birational map $S_5 \dashrightarrow S_5$. 
	\end{enumerate}
\end{Lemma}

\begin{proof}
	It is clear that we can obtain case \emph{a)} if we blow up the five points in $\mathcal{K}_1$ or $\mathcal{K}_2$, i.e. $\tau = \sigma$ in the link \eqref{MainDiagram}. We get back exactly the model of the Clebsch cubic we had before because the elements in $\mathcal{K}_1$ or $\mathcal{K}_2$ are the points we obtained by the contraction described in \autoref{Lem_LinkClebschP1xP1}. For symmetric reasons we may assume that these are the points in $\mathcal{K}_1$. Therefore  $\pi \circ \tau^{*} = \left(\sigma \circ \pi^*\right)^{-1}=\psi^{-1}$ as described in \eqref{MainDiagram}.\\
	\autoref{Prop_OrbitsP1xP1} tells us that the two orbits $\mathcal{K}_1$ and $\mathcal{K}_2$ are interchanged by an automorphism. Let us now consider the blow up $\tau: \Pro^1 \times \Pro^1 \to \widetilde{S}$ of the orbit $\mathcal{K}_2$, which is not the same blow up as $\sigma$. Then we may contract the two $(-1)$-curves, $E_1$ and $E_2$, on the Clebsch cubic.\\
	This gives us back $S_5$ because the smooth quintic del Pezzo surface is unique. This means that $\phi \circ \tau^* \circ\pi$ is a birational map $S_5 \dashrightarrow S_5$. We obtain that $\psi \circ \tau^* \circ \pi=\psi \circ \phi =\chi: S_5 \dashrightarrow S_5$ as shown in \eqref{Fig_Blow_Up_orbit_5_P1_P1}. This is a birational map $S_5 \dashrightarrow S_5$ which is not biregular.
\end{proof}

	\begin{align}
		\tag{$\clubsuit$}
		\label{Fig_Blow_Up_orbit_5_P1_P1}
		\begin{split}
	\xymatrix{
		&\widetilde{S}\ar@{->}[ld]_{\pi}\ar@{->}[dr]^{\sigma}&& \widetilde{S}\ar@{->}[ld]_{\tau}\ar@{->}[dr]^{\pi}\\
		S_5\ar@{-->}[rr]_{\psi}&&\mathbb{P}^1\times\mathbb{P}^1\ar@{-->}[rr]_{\phi}&&S_5}
	\end{split}
	\end{align}

\autoref{Lem_BlowupOrbit4P1xP1} and \autoref{Lem_Blowuporbit5P1xP1} tell us that there is no $G_{20}$-equivariant link starting from $\Pro^1 \times \Pro^1$, that leads to a different minimal surface than the quintic del Pezzo surface or $\Pro^1 \times \Pro^1$ itself. This together with \autoref{Prop_Picard_Z2_Clebsch} finalises the proof of \autoref{MaintechnicalResult} and implies the first two parts of \autoref{Main_theorem}.\\

In \autoref{MaintechnicalResult} we additionally stated that $\mathrm{Bir}^{G_{20}}(S_5)$ is of order 40. In fact one can show that 

\begin{Theorem}
	\label{Thm_G40}
	Let $S_5$ be the smooth del Pezzo surface of degree 5, and let $G_{20}\cong C_5 \rtimes C_4$ be a subgroup of order 20 in  $\Aut\left(S_5\right)$. Then
	$$
	\mathrm{Bir}^{G_{20}}\left(S_5\right)=G_{40},
	$$
	where $G_{40} \cong C_2 \times G_{20}$.
\end{Theorem}

\begin{proof}
	We need to find the normalizer $G_{40}=\operatorname{Norm}_{\Aut(\Pro^{1}\times \Pro^{1})}(G_{20})$. Obviously, it is enough to find $G_{40} \cap H$, where $H$ is the subgroup of $\Aut(\Pro^{1}\times \Pro^{1})$ which preserves rulings. Certainly, $G_{40} \cap H$  lies inside the group $\operatorname{Norm}_{H} (D_{10})$. The normalizer of $D_{10}$ in $\Aut(\Pro^{1})$ is equal to $D_{20}$ and generated by $D_{10}$ and the involution $[x:y]\mapsto[-x:y]$. Thus $G_{40}\cap H$ lies inside the group $<D_{10}, a, b>$, with 
	\begin{align*}
	&a:\left([x_1:y_1],[x_2:y_2]\right) \mapsto \left([-x_1:y_1],[x_2:y_2]\right)\\ &b: \left([x_1:y_1],[x_2:y_2]\right)\mapsto\left([x_1:y_1],[-x_2: y_2]\right).
\end{align*}
 One can easily check that only $ab$ normalizes the group $G_{20}$ and $G_{40}\cong C_{2}\times G_{20}$.

\end{proof}

This proof was communicated to me by Artem Avilov and I thank him for thus completing the proof of \autoref{Main_theorem}.








\newpage

\bibliography{References}{}
\bibliographystyle{plain}
\end{document}